\documentclass[11pt]{article}

\usepackage{amsmath,amssymb,amsfonts,amsthm}
\usepackage{hyperref}

\usepackage[font={small,sl}]{caption}
\newtheorem{theorem}{Theorem}[section]
\newtheorem{lemma}[theorem]{Lemma}
\newtheorem{corollary}[theorem]{Corollary}
\theoremstyle{definition}

\newtheorem{proposition}[theorem]{Proposition}

\theoremstyle{remark}

\numberwithin{equation}{section}

\makeatletter
\def\blfootnote{\xdef\@thefnmark{}\@footnotetext}
\newcommand*{\transpose}{%
  {\mathpalette\@transpose{}}%
}
\newcommand*{\@transpose}[2]{%
  \raisebox{\depth}{$\m@th#1\intercal$}%
}
\makeatother

\newcommand{\C}{\mathbb{C}}
\newcommand{\R}{\mathbb{R}}

\hypersetup{
  hidelinks,
  pdftitle={Quantum entropy and the Bannai--Ito multiplicity conjecture},
  pdfauthor={Gary Greaves, Haoran Zhu},
  pdfkeywords={association schemes, Q-polynomial ordering, multiplicities, log-concavity, Krein parameters}
}

\DeclareMathOperator{\Symm}{Sym}
\DeclareMathOperator{\Span}{span}
\DeclareMathOperator{\Tr}{Tr}
\DeclareMathOperator{\rank}{rank}

\DeclareMathOperator{\Mat}{Mat}

\newcommand{\cA}{\mathcal A}
\newcommand{\XX}{\mathfrak X}

\usepackage{mathtools}
\usepackage[all]{xy}
\usepackage{authblk}

\begin{document}

\title{Quantum entropy and the \\ Bannai--Ito multiplicity conjecture}

\author[1]{Gary Greaves}
\author[1]{Haoran Zhu}

\affil[1]{%
  Division of Mathematical Sciences,
  School of Physical and Mathematical Sciences,
  Nanyang Technological University,
  21 Nanyang Link, Singapore 637371\\
  \href{mailto:gary@ntu.edu.sg}{\texttt{gary@ntu.edu.sg}}
  \qquad
  \href{mailto:zhuh0031@e.ntu.edu.sg}{\texttt{zhuh0031@e.ntu.edu.sg}}
}



\date{}


\maketitle

\begin{abstract}
We prove that the multiplicities of every symmetric $Q$-polynomial association scheme form a log-concave sequence.
This resolves the 1984 unimodality conjecture of Bannai and Ito and establishes both related inequalities conjectured by Stanton.
Our proof is a novel application of strong subadditivity and weak monotonicity inequalities for quantum entropy.
\end{abstract}

\section{Introduction}

A finite connected graph is \textbf{distance-regular} if, for any vertices $x,y$, the number of vertices at distance $i$ from $x$ and distance $j$ from $y$ depends only on $i,j$ and the distance
between $x$ and $y$.
Distance-regular graphs arise naturally in coding theory, design theory, and finite geometry.
Classical families include the Hamming, Johnson, Grassmann, and dual polar graphs; see the survey of van Dam, Koolen,
and Tanaka~\cite{vanDamKoolenTanaka}.

A \textbf{symmetric association scheme} can be viewed as an abstraction of a distance-regular graph by replacing distance classes with a partition of the pairs of a finite set into symmetric relations, with equality as one relation, while retaining the same uniform intersection-counting property.
Each relation is represented by a $\{0,1\}$-matrix $A_i$.
For a distance-regular graph of diameter $d$, these are the distance matrices $A_0,A_1,\ldots,A_d$, where $[A_i]_{x,y}=1$
precisely when $x$ and $y$ are at distance $i$.
Distance-regular graphs can be characterised as symmetric association schemes whose relation matrices can be ordered such that $A_i$ is a polynomial of degree $i$ in the adjacency matrix $A_1$.
This characteristic of distance-regular graphs among symmetric association schemes is called the \textbf{$P$-polynomial property}.

The relation matrices span the \textbf{Bose--Mesner algebra}, which is closed under both ordinary matrix multiplication and entrywise multiplication.
Since the relation matrices are real symmetric and commute, they have an orthogonal decomposition into common eigenspaces.
The projections onto these eigenspaces, called the \textbf{primitive idempotents}, form another basis
$E_0,E_1,\ldots,E_d$ of the same algebra.
The \textbf{$Q$-polynomial property} is obtained by exchanging the two bases and the two products, that is, in a suitable ordering, each $E_i$ can be expressed as a polynomial of degree $i$ in $E_1$ with ordinary multiplication replaced by entrywise multiplication.
In this sense, $Q$-polynomial association schemes are the algebraic dual counterparts of distance-regular graphs; see \cite{BannaiIto,BrouwerCohenNeumaier,Delsarte,MartinTanaka}.

Under this duality, the valencies $k_i$, which count vertices in the distance layers of a distance-regular graph,
correspond to the multiplicities $m_i=\rank(E_i)$, which measure the dimensions of the common eigenspaces.
This correspondence motivates the question of whether the multiplicities satisfy the same inequalities as the valencies.

For a distance-regular graph, Taylor and Levingston proved that the valencies are unimodal and that, moreover, for each $i \in \{0,\dots,\lceil d/2\rceil-1\}$,
\[
 k_i\leqslant k_{i+1}
 \quad\text{and}\quad
 k_i\leqslant k_{d-i};
\]
see \cite{TaylorLevingston} and \cite[Theorem~1.1]{SaganCaughman}. 
Writing $b_i,c_i$ for the usual distance-regular graph intersection numbers, the recurrence $k_ib_i=k_{i+1}c_{i+1}$, together with $b_{i-1}\geqslant b_i$ and $c_i\leqslant c_{i+1}$, implies
\[
 k_i^2\geqslant k_{i-1}k_{i+1}.
\]
Thus, the valencies form a log-concave sequence.

The dual problem concerns the multiplicities of a \textbf{$Q$-polynomial association scheme}. Let $E_0,E_1,\ldots,E_d$ be a \textbf{$Q$-polynomial ordering} of the primitive idempotents, and write $m_i=\rank(E_i)$. 
Bannai and Ito conjectured that the sequence $m_0,m_1,\ldots,m_d$ is unimodal \cite[p.~205]{BannaiIto}; see also \cite[Conjecture~1.2]{SaganCaughman}. 
Stanton conjectured \cite[Conjecture~1.3]{SaganCaughman}  that, for each $i \in \{0,\dots,\lceil d/2\rceil-1\}$,
\[
 m_i\leqslant m_{i+1}
 \quad\text{and}\quad
 m_i\leqslant m_{d-i}.
\]
Sagan and Caughman proved these inequalities for dual-thin schemes \cite[Theorem~1.4]{SaganCaughman}. Pascasio proved them for $Q$-polynomial distance-regular graphs \cite[p.~1073]{Pascasio}, using the symmetry and unimodality of tridiagonal pairs \cite[Corollaries~5.7 and~6.6]{ItoTanabeTerwilliger}. Log-concavity also follows at once for $Q$-polynomial translation schemes, since their multiplicities are the valencies of the dual $P$-polynomial schemes; see \cite[Section~2.10]{BrouwerCohenNeumaier}. Martin and Tanaka record the two problems in \cite[Conjectures~4.3 and~4.4]{MartinTanaka}.
 
More recently, Shi, Wang and Sol\'e proved log-concavity under a monotonicity assumption on the Krein array \cite[Theorem~5]{ShiWangSole}. 
They also established it directly for the Johnson and Grassmann schemes, the folded Johnson graphs, and the Odd graphs \cite[Theorems~6--7 and Corollaries~6--7]{ShiWangSole}.

Our main result is the following.

\begin{theorem}\label{thm:main}
Let $(\XX,\mathfrak R)$ be a symmetric $d$-class association scheme, and let $E_0,E_1,\ldots,E_d$ be a $Q$-polynomial ordering of its primitive idempotents. 
For $0\leqslant i\leqslant d$, write $m_i=\rank(E_i)$. 
Then, for all non-negative integers $a,b,c$ with $a+b+c\leqslant d$,
\begin{align}
 m_{a+b}m_{b+c}&\geqslant m_bm_{a+b+c}, \label{eq:intro-ssa}\\
 m_{a+b}m_{a+c}&\geqslant m_bm_c. \label{eq:intro-wm}
\end{align}
\end{theorem}

Theorem~\ref{thm:main} extends Pascasio's multiplicity inequalities \cite{Pascasio} to all symmetric $Q$-polynomial association schemes and subsumes the multiplicity log-concavity results of Shi, Wang and Sol\'e \cite{ShiWangSole}, without any monotonicity assumption on the Krein array.

The technique of our proof of Theorem~\ref{thm:main} is to construct, for each $0\leqslant t\leqslant d$, a real symmetric idempotent matrix $\mathsf P_t$ of rank $m_t$. 
The normalised matrices are positive semidefinite matrices of trace one, and taking a partial trace produces an earlier member of the family. 
The two inequalities in Theorem~\ref{thm:main} then follow from strong subadditivity and weak monotonicity of quantum entropy~\cite[(3.1) and (3.2)]{Pippenger}.

Taking $a=c=1$ and $b=i-1$ in \eqref{eq:intro-ssa} shows that the multiplicities form a log-concave sequence.
Since each $m_i$ is positive, the sequence $m_0,m_1,\ldots,m_d$ is unimodal; whence we obtain the Bannai--Ito multiplicity conjecture.

\begin{corollary}[Bannai--Ito multiplicity conjecture]\label{cor:bannai-ito}
Let $E_0,E_1,\ldots,E_d$ be a $Q$-polynomial ordering of the primitive idempotents of a symmetric association scheme, and write $m_i=\rank(E_i)$ for $0\leqslant i\leqslant d$. Then the sequence $m_0,m_1,\ldots,m_d$ is unimodal.
\end{corollary}

Taking $a=j-i$ and $b=c=i$ in \eqref{eq:intro-wm} reduces to $m_i\leqslant m_j$ whenever $0\leqslant i\leqslant j$ and $i+j\leqslant d$.
This proves Stanton’s conjecture.

\begin{corollary}[Stanton's conjecture]\label{cor:stanton}
Let $E_0,E_1,\ldots,E_d$ be a $Q$-polynomial ordering of the primitive idempotents of a symmetric association scheme, and write $m_i=\rank(E_i)$ for $0\leqslant i\leqslant d$. Then, for each $i \in \{0,\dots,\lceil d/2\rceil-1\}$, we have
\[
 m_i\leqslant m_{i+1}
 \quad\text{and}\quad
 m_i\leqslant m_{d-i}.
\]
\end{corollary}

Theorem~\ref{thm:main} also provides a simple nonexistence criterion for proposed parameters of $Q$-polynomial association schemes.
In particular, it gives short proofs of six known nonexistence results recorded in Williford's table of four-class $Q$-bipartite association schemes~\cite{WillifordQbipartite4} (see also \cite{GavrilyukVidaliWilliford}).
These are the parameter sets indexed by
\[
(v,m_1)\in
\{(594,9),(7776,27),(8432,31),(8478,27),
  (9984,24),(9984,32)\},
\]
whose nonexistence is attributed to Kodalen and Martin~\cite[p.~55]{Kodalen}.
In every case, the proposed multiplicities satisfy
$m_2^2<m_1m_3$, which violates
\eqref{eq:intro-ssa} with $a=b=c=1$.

The paper is organised as follows. 
In Section~\ref{sec:schemes}, we recall relevant facts about $Q$-polynomial association schemes and their Krein parameters. 
In Section~\ref{sec:ranks}, for each $t \in \{0,\dots,d\}$, we define projection matrices $\mathsf P_t$ having rank $m_t$.
We show that these matrices satisfy a telescoping partial-trace relation.
Finally, we apply the quantum entropy inequalities of Lieb and Ruskai~\cite{LiebRuskai} to prove Theorem~\ref{thm:main}.

\section{\texorpdfstring{$Q$-polynomial}{Q-polynomial} association schemes}\label{sec:schemes}

We recall the notation and standard facts about $Q$-polynomial association schemes used below; for the standard theory, see \cite{BannaiIto,BrouwerCohenNeumaier,Delsarte,MartinTanaka}.

Let $\XX$ be a finite set of cardinality $v$, and let $\mathfrak R=\{R_0,R_1,\ldots,R_d\}$ be a \textbf{partition} of $\XX\times\XX$. Suppose that $R_0$ is the diagonal relation, each $R_i$ is symmetric, and, whenever $(x,y)\in R_k$, the number of points $z$ such that $(x,z)\in R_i$ and $(z,y)\in R_j$ depends only on $i,j,k$. Then $(\XX,\mathfrak R)$ is a \textbf{symmetric association scheme} with $d$ classes.

For finite sets $\mathfrak X$ and $\mathfrak Y$ and a field $\mathbb F$, we write $\Mat_{\mathfrak X,\mathfrak Y}(\mathbb F)$ for the set of matrices over $\mathbb F$ whose rows are indexed by $\mathfrak X$ and whose columns are indexed by $\mathfrak Y$. We abbreviate $\Mat_{\mathfrak X}(\mathbb F)=\Mat_{\mathfrak X,\mathfrak X}(\mathbb F)$.

For $0\leqslant i\leqslant d$, let $A_i\in\Mat_{\XX}(\R)$ be the $0$-$1$ adjacency matrix of $R_i$. The \textbf{Bose--Mesner algebra} of $(\XX,\mathfrak R)$ is $\cA=\Span_{\C}\{A_0,A_1,\ldots,A_d\}\subseteq\Mat_{\XX}(\C)$, a commutative matrix algebra with real symmetric standard basis $A_0,A_1,\ldots,A_d$; in particular, it contains the identity matrix $I=A_0$ and the all-ones matrix $J=\sum_{i=0}^dA_i$. Let $E_0,E_1,\ldots,E_d$ be its primitive idempotents, ordered so that $E_0=v^{-1}J$. Then
\[
 E_iE_j=\delta_{ij}E_i
 \quad\text{and}\quad
 \sum_{i=0}^dE_i=I.
\]
Here $\delta_{ij}$ is the Kronecker delta. In particular, the $E_i$ are non-zero real symmetric projection matrices: $E_i^\transpose=E_i^2=E_i$. We write $m_i=\rank(E_i)$ for the \textbf{multiplicity} of $E_i$; thus $m_i$ is positive, every diagonal entry of $E_i$ is equal to $m_i/v$, and $m_0=1$.

The algebra $\cA$ is also closed under the entrywise product $\circ$, defined by $(M\circ N)_{xy}=M_{xy}N_{xy}$. 
For each $i,j \in \{0,\dots,d\}$, we can write
\begin{equation}\label{eq:Krein-expansion}
 E_i\circ E_j=\frac{1}{v}\sum_{k=0}^d q_{ij}^kE_k.
\end{equation}
The coefficients $q_{ij}^k$ are called the \textbf{Krein parameters}.
It is well known that the Krein parameters are non-negative real numbers
\cite{MartinTanaka}.

For later use, we record the following standard identities for the Krein
parameters.
\begin{proposition}[{\cite[Lemma~2.3.1(iv)]{BrouwerCohenNeumaier},
  \cite[Theorem~11.4.1(i),(iii)]{BrouwerHaemers}}]
  \label{prop:Krein}
Let $0\leqslant i,j,k\leqslant d$. Then
\[
 q_{ij}^k=q_{ji}^k
 \qquad\text{and}\qquad
 q_{ij}^k m_k=q_{ik}^j m_j.
\]
\end{proposition}

An ordering $E_0,E_1,\ldots,E_d$ is called \textbf{$Q$-polynomial} if the matrix
\[
 (q_{1j}^k)_{0\leqslant j,k\leqslant d}
\]
is tridiagonal and irreducible; see \cite[Section~4]{MartinTanaka}.
Equivalently, $q_{1j}^k=0$ when $|j-k|>1$, and $q_{1j}^{j-1}$ and $q_{1j}^{j+1}$ are positive whenever the indicated indices lie between $0$ and $d$.

We shall use the following standard triangle property repeatedly.
\begin{proposition}[{\cite[Theorem~5.16]{Delsarte}}]
  \label{prop:triangle}
Let $E_0,E_1,\ldots,E_d$ be a $Q$-polynomial ordering, and let $0\leqslant i,j,k\leqslant d$. 
Then
\[
 q_{ij}^k=0
 \qquad\text{unless}\qquad
 |i-j|\leqslant k\leqslant i+j.
\]
Moreover, $q_{ij}^{i+j}>0$ whenever $i+j\leqslant d$.
\end{proposition}

\section{Projection matrices and multiplicity inequalities}\label{sec:ranks}

In this section, we construct projection matrices whose ranks are the multiplicities and use their partial traces to prove Theorem~\ref{thm:main}. All matrices in this section have real entries, and we accordingly use the transpose throughout; $\otimes$ denotes the Kronecker product.

\subsection{Construction of the projection matrices}

We now construct the projection matrices used to prove Theorem~\ref{thm:main}. For a positive integer $t$, let $D_t\in\Mat_{\XX^t,\XX}(\R)$ be the matrix with
\[
 [D_t]_{x_1,\ldots,x_t,y}=
 \begin{cases}
  1,&\text{if }x_1=\cdots=x_t=y,\\
  0,&\text{otherwise.}
 \end{cases}
\]
Thus, $D_1=I$.

\begin{lemma}\label{lem:entrywise-product}
Let $t$ be a positive integer and let $M_1,\ldots,M_t\in\Mat_{\XX}(\R)$. Then
\[
 D_t^\transpose(M_1\otimes\cdots\otimes M_t)D_t
 =M_1\circ\cdots\circ M_t.
\]
\end{lemma}

\begin{proof}
For $x,y\in\XX$, the column of $D_t$ indexed by $x$ is the standard basis vector of $\R^{\XX^t}$ indexed by the constant tuple $(x,\ldots,x)$. Hence the $(x,y)$-entry of the left-hand side is the entry of $M_1\otimes\cdots\otimes M_t$ in row $(x,\ldots,x)$ and column $(y,\ldots,y)$, which is $\prod_{i=1}^t(M_i)_{xy}$.
\end{proof}

For $t=2$, Lemma~\ref{lem:entrywise-product} is the standard realisation of the Hadamard product as the compression of the Kronecker product to the diagonal tensor subspace; see \cite[p.~31]{Paulsen}. The general identity is its immediate multi-factor extension. In the association-scheme setting, see also \cite[p.~1501]{MartinTanaka}.

Let $(\XX,\mathfrak R)$ be a symmetric association scheme with $d$ classes, and fix a $Q$-polynomial ordering $E_0,E_1,\ldots,E_d$ of its primitive idempotents. 
Write $m_i=\rank(E_i)$.
For $1\leqslant t\leqslant d$, let $E_1^{\otimes t}$ (resp.\ $E_1^{\circ t}$) denote the $t$-fold Kronecker (resp.\ entrywise) product of $E_1$ with itself.

We now state and prove a key application of Lemma~\ref{lem:entrywise-product} to primitive idempotents of an association scheme.

\begin{corollary}\label{cor:F-norm}
Let $t\in\{1,\ldots,d\}$.
Suppose $E_1^{\circ t}=\sum_{k=0}^dc^{(t)}_kE_k$. Then $c^{(t)}_k\geqslant0$ for all $k$, and $c^{(t)}_k=0$ for $k>t$. 
Moreover, $c^{(t)}_t > 0$ and
\[
 E_tD_t^\transpose\bigl(E_1^{\otimes t}\bigr)D_tE_t=c^{(t)}_t E_t.
\]
\end{corollary}

\begin{proof}
By Lemma~\ref{lem:entrywise-product} and the orthogonality of the primitive idempotents, we have
\[
E_tD_t^\transpose\bigl(E_1^{\otimes t}\bigr)D_tE_t
 =E_t\bigl(E_1^{\circ t}\bigr)E_t
 =c^{(t)}_tE_t.
\]

The remaining assertions follow by induction on $t$. The case $t=1$ is immediate. 
Suppose $t\geqslant2$.
Applying \eqref{eq:Krein-expansion}, we find that
\[
\sum_{k=0}^dc^{(t)}_kE_k = E_1^{\circ t}
=E_1\circ E_1^{\circ(t-1)} = \sum_{k=0}^d
  \left(\frac1v\sum_{j=0}^d c_j^{(t-1)}q_{1j}^k\right)E_k.
\]
Compare coefficients to obtain
\begin{equation}\label{eq:c-recursion}
 c_k^{(t)}
 =\frac1v\sum_{j=0}^d c_j^{(t-1)}q_{1j}^k.
\end{equation}
Since the Krein parameters $q_{1j}^k$ are nonnegative, we have $c_k^{(t)}\geqslant0$. 
By induction,
$c_j^{(t-1)}=0$ for $j>t-1$. 
By Proposition~\ref{prop:triangle}, if $j\leqslant t-1$ and $k>t$ then $q_{1j}^k=0$, which implies $c_k^{(t)}=0$.
Again, by Proposition~\ref{prop:triangle}, if $k=t$ then $q_{1j}^t=0$ for $j<t-1$ and $q_{1,t-1}^t>0$. Therefore,
\[
 c_t^{(t)}
 =\frac1v c_{t-1}^{(t-1)}q_{1,t-1}^{t}>0.
 \qedhere
\]
\end{proof}

Motivated by Corollary~\ref{cor:F-norm}, for each $t \in \{1,\dots,d\}$, we define 
\[
 \mathsf P_t
 :=\alpha_t^{-1}F_tF_t^\transpose
 \in\Mat_{\XX^t}(\R),
\]
where
\[
 F_t:=E_1^{\otimes t}D_tE_t\in\Mat_{\XX^t,\XX}(\R)
\]
and $\alpha_t \coloneq c^{(t)}_{t}$.
 
Note that $\mathsf P_1=F_1=E_1$ and $\alpha_1 = 1$.
Since
\[
 \mathsf P_t^2=\alpha_t^{-2}F_t\bigl(F_t^\transpose F_t\bigr)F_t^\transpose
 =\alpha_t^{-1}F_tE_tF_t^\transpose=\mathsf P_t,
\]
the matrix $ \mathsf P_t$ is idempotent, and hence, since it is also symmetric, it is a projection matrix.
Moreover, by Corollary~\ref{cor:F-norm},
\begin{equation}
    \label{eqn:rank}
 \rank(\mathsf P_t)=\Tr(\mathsf P_t)
 =\alpha_t^{-1}\Tr(F_t^\transpose F_t)
 =\Tr(E_t)=m_t.
\end{equation}
For $t=0$, we define $\mathsf P_0:=1\in\Mat_{\XX^0}(\R)$ (we consider $\mathfrak X^0$ to be a singleton).
By the proof of Corollary~\ref{cor:F-norm}, for each $t \in \{2,\dots,d\}$,
\begin{equation}\label{eq:alpha-recursion}
 \alpha_t
 =\frac{q_{1,t-1}^{t}}{v}\alpha_{t-1}.
\end{equation}

For a positive integer $t$, let $\Symm(t)$ act on $\XX^t$ by
\[
 \sigma(x_1,\ldots,x_t)
 :=(x_{\sigma^{-1}(1)},\ldots,x_{\sigma^{-1}(t)})
\]
for each $\sigma \in \Symm(t)$.
We say that $K\in\Mat_{\XX^t}(\R)$ is \textbf{coordinate invariant} if, for all $\sigma\in\Symm(t)$ and
$\mathbf x,\mathbf y\in\XX^t$, we have
\begin{equation}\label{eq:coordinate-permutation-invariance}
 K_{\sigma\mathbf x,\,\sigma\mathbf y}=K_{\mathbf x,\mathbf y}.
\end{equation}

\begin{lemma}\label{lem:Ptinv}
For each $t \in \{0,\dots,d\}$, the matrix $\mathsf P_t$ is coordinate invariant.
\end{lemma}

\begin{proof}
The case for $t = 0$ is trivial.
    For $t\geqslant1$ and $\mathbf x\in\XX^t$ and $z\in\XX$, we have
\[
 [F_t]_{\mathbf x,z}
 =\sum_{u\in\XX}\left(\prod_{i=1}^t[E_1]_{\mathbf x_i,u}\right)[E_t]_{u,z},
\]
and hence $[F_t]_{\sigma\mathbf x,z}=[F_t]_{\mathbf x,z}$ for every $\sigma\in\Symm(t)$.
Since $\mathsf P_t=\alpha_t^{-1}F_tF_t^\transpose$, it follows that $\mathsf P_t$ is coordinate invariant.
\end{proof}

\subsection{Partial traces}

Let $\mathfrak X$ and $\mathfrak Y$ be finite sets, and let
$K\in\Mat_{\mathfrak X\times\mathfrak Y}(\R)$. 
For $x,y\in\mathfrak X$, let
\[
 K_{xy}:=\bigl[K_{(x,z),(y,w)}\bigr]_{z,w\in\mathfrak Y}
 \in\Mat_{\mathfrak Y}(\R).
\]
Thus, $K=[K_{xy}]_{x,y\in\mathfrak X}$ is an
$\mathfrak X\times\mathfrak X$ block matrix whose blocks are indexed by $\mathfrak Y$. 
The \textbf{partial trace of $K$ over $\mathfrak Y$} is
\begin{equation}\label{eq:partial-trace-general}
 \Tr_{\mathfrak Y}(K)
 :=\bigl[\Tr(K_{xy})\bigr]_{x,y\in\mathfrak X}
 \in\Mat_{\mathfrak X}(\R),
\end{equation}
where $\Tr$ denotes the ordinary matrix trace. 
This is the block-matrix definition of the partial trace in \cite[Proposition~4.3.10]{Bhatia}.

We extend this notation to products of more than two sets. 
Let $n$ be a positive integer, write $[n]:=\{1,\ldots,n\}$, and let $\mathfrak X_1,\ldots,\mathfrak X_n$ be finite sets. 
For $J\subseteq[n]$, define
\[
 \mathfrak X_J:=\prod_{i\in J}\mathfrak X_i,
\]
where the factors occur in their natural order, and let $\mathfrak X_\emptyset$ be a singleton. 
We regard $K\in\Mat_{\mathfrak X_{[n]}}(\R)$ as a block matrix in $\Mat_{\mathfrak X_{[n]\setminus J}\times\mathfrak X_J}(\R)$.
We write $\Tr_J(K)$ for its partial trace over $\mathfrak X_J$.

For non-negative integers $r$ and $s$, take $n=r+s$, set $\mathfrak X_i=\mathfrak X$ for
$1\leqslant i\leqslant n$, and $J=\{r+1,\ldots,r+s\}$. 
We abbreviate $\Tr_J$ by $\Tr_s$.
Thus, for $\mathbf x,\mathbf y\in\mathfrak X^r$, we have
\begin{equation}\label{eq:partial-trace}
 \bigl[\Tr_s(K)\bigr]_{\mathbf x,\mathbf y}
 =\sum_{\mathbf z\in\mathfrak X^s}
 K_{(\mathbf x,\mathbf z),(\mathbf y,\mathbf z)}.
\end{equation}
In particular, $\Tr_0$ is the identity, while $\Tr_t$ on $\Mat_{\mathfrak X^t}(\R)$ is the ordinary matrix trace $\Tr$.
More generally, if $a,b$ are non-negative integers with $a+b\leqslant t$ and $K\in\Mat_{\mathfrak X^t}(\R)$, then
\begin{equation}\label{eq:successive-partial-traces}
 \Tr_a\bigl(\Tr_b(K)\bigr)=\Tr_{a+b}(K),
\end{equation}
where the ambient matrix spaces of the two partial traces are understood from context. Indeed, for
$\mathbf x,\mathbf y\in\mathfrak X^{t-a-b}$,
\[
 \bigl[\Tr_a(\Tr_b(K))\bigr]_{\mathbf x,\mathbf y}
 =
 \sum_{\mathbf u\in\mathfrak X^a}
 \sum_{\mathbf z\in\mathfrak X^b}
 K_{(\mathbf x,\mathbf u,\mathbf z),
   (\mathbf y,\mathbf u,\mathbf z)}
 =
 \bigl[\Tr_{a+b}(K)\bigr]_{\mathbf x,\mathbf y}.
\]

Now we can prove a key recursive telescoping relation with the partial trace and the projections $\mathsf P_t$.

\begin{theorem}\label{thm:compatible-projections}
Let $r$ and $t$ be integers with $0\leqslant r\leqslant t\leqslant d$. 
Then
\[
 \Tr_{t-r}(\mathsf P_t)=\frac{m_t}{m_r}\mathsf P_r.
\]
\end{theorem}

\begin{proof}
The cases $r=t$ and $r=0$ follow from the definition of partial trace and $\Tr(\mathsf P_t)=m_t$.
It suffices to prove, for each $t \in \{2,3,\dots,d\}$, that
\[
 \Tr_1(\mathsf P_t)
 =\frac{m_t}{m_{t-1}}\mathsf P_{t-1}
\]
since the general case follows by iterating
\eqref{eq:successive-partial-traces}.

Fix $2\leqslant t\leqslant d$ and let $r=t-1$, and set $G_r:=E_1^{\otimes r}D_r$.
Then $F_r = G_rE_r$. 
By Lemma~\ref{lem:entrywise-product},
\[
 (G_rE_k)^\transpose(G_rE_k)
 =E_k(E_1^{\circ r})E_k
 =c_k^{(r)}E_k.
\]
By Corollary~\ref{cor:F-norm}, we have $G_rE_k=0$ whenever $k>r$.
For $\mathbf x,\mathbf y\in\XX^r$, expanding $F_t$ and using the idempotency and symmetry of $E_1$ and $E_t$ yields
\begin{align*}
 \bigl[\Tr_1(F_tF_t^\transpose)\bigr]_{\mathbf x,\mathbf y}
 &=
 \sum_{u,w\in\XX}
 [G_r]_{\mathbf x,u}[G_r]_{\mathbf y,w}[E_t]_{u,w}
 \sum_{z\in\XX}[E_1]_{z,u}[E_1]_{z,w}\\
 &=
 \bigl[G_r(E_1\circ E_t)G_r^\transpose\bigr]_{\mathbf x,\mathbf y}.
\end{align*}

By Proposition~\ref{prop:triangle},
$q_{1t}^k=0$ for $k<r$, since $r=t-1=|1-t|$.
Since $G_rE_k=0$ for $k>r$, we can use \eqref{eq:Krein-expansion} to obtain
\[
 G_r(E_1\circ E_t)G_r^\transpose
 =\frac{q_{1t}^r}{v}G_rE_rG_r^\transpose
 =\frac{q_{1t}^r}{v}F_rF_r^\transpose =\frac{\alpha_r q_{1t}^r}{v}\mathsf P_r.
\]
Finally, we apply \eqref{eq:alpha-recursion} and
Proposition~\ref{prop:Krein} to deduce
\[
 \Tr_1(\mathsf P_t) = \frac{1}{\alpha_t}\Tr_1(F_tF_t^\transpose)
 =
 \frac{\alpha_r q_{1t}^r}{v\alpha_t}\mathsf P_r
=
 \frac{m_t}{m_r}\mathsf P_r,
\]
as required.
\end{proof}

\subsection{Density matrices and entropy}

A \textbf{density matrix} is a positive semidefinite matrix of trace one. 
Now we show that the partial trace of a density matrix is again a density matrix.

\begin{lemma}\label{lem:partial-trace-basic}
Let $\mathfrak X$ and $\mathfrak Y$ be finite sets and let $K\in\Mat_{\mathfrak X\times\mathfrak Y}(\R)$.
If $K$ is positive semidefinite, then $\Tr_{\mathfrak Y}(K)$ is positive semidefinite. 
Moreover,
\[
 \Tr\bigl(\Tr_{\mathfrak Y}(K)\bigr)=\Tr(K).
\]
\end{lemma}

\begin{proof}
Let $\{\mathbf e_z\; : \; z\in\mathfrak Y\}$ be the standard basis of $\R^{\mathfrak Y}$. 
For $\mathbf u\in\R^{\mathfrak X}$, \eqref{eq:partial-trace-general} implies
\[
 \mathbf u^\transpose\Tr_{\mathfrak Y}(K)\mathbf u
 =
 \sum_{z\in\mathfrak Y}
 (\mathbf u\otimes\mathbf e_z)^\transpose
 K(\mathbf u\otimes\mathbf e_z),
\]
which is non-negative when $K$ is positive semidefinite.
The trace identity follows since
\[
 \Tr\bigl(\Tr_{\mathfrak Y}(K)\bigr)
 =
 \sum_{x\in\mathfrak X}\sum_{z\in\mathfrak Y}
 K_{(x,z),(x,z)}
 =
 \Tr(K). \qedhere
\]
\end{proof}

If $M$ is a density matrix with eigenvalues $\lambda_1,\ldots,\lambda_n$, its \textbf{von Neumann entropy} is defined as
\[
 h(M):=-\sum_{i=1}^n\lambda_i\log\lambda_i,
\]
where $\log$ denotes the natural logarithm and $0\log0$ is interpreted as zero. 
Since each eigenvalue $\lambda_i$ satisfies $0\leqslant\lambda_i\leqslant1$, the entropy $h(M)$ is a non-negative real number.

\begin{proof}[Proof of Theorem~\ref{thm:main}]
Let $a,b,c$ be non-negative integers with $t:=a+b+c\leqslant d$.
The case $t=0$ is immediate, so suppose that $t\geqslant1$.
Set
\[
 \mathfrak X_1:=\mathfrak X^a,\qquad
 \mathfrak X_2:=\mathfrak X^b,\qquad
 \mathfrak X_3:=\mathfrak X^c.
\]
We consider $\mathsf P_t$ as a matrix with rows and columns indexed by $\mathfrak X_1\times\mathfrak X_2\times\mathfrak X_3$ under its natural bijection with $\mathfrak X^t$.

For $J\subseteq\{1,2,3\}$, write
$\overline{J}:=\{1,2,3\}\setminus J$ and
\[
 \mathfrak X_J:=\prod_{j\in J}\mathfrak X_j,
\]
where the factors $\mathfrak X_j$ occur in their natural order and $\mathfrak X_\emptyset$ is a singleton.
Observe that $m_t^{-1}\mathsf P_t$ is a density matrix.
Define
\[
 S_J:=h\bigl(\Tr_{\overline{J}}(m_t^{-1}\mathsf P_t)\bigr).
\]
Lieb and Ruskai (see \cite[Theorem~2(i),(ii)]{LiebRuskai}
and \cite[(3.1) and (3.2)]{Pippenger}) showed that
\begin{align}
 S_{\{1,2,3\}}+S_{\{2\}}
 &\leqslant S_{\{1,2\}}+S_{\{2,3\}},
 \label{eq:strong-subadditivity}\\
 S_{\{2\}}+S_{\{3\}}
 &\leqslant S_{\{1,2\}}+S_{\{1,3\}}.
 \label{eq:weak-monotonicity}
\end{align}

Fix $J\subseteq\{1,2,3\}$, and let $j$ be the total number of individual coordinates in the blocks indexed by $J$.
Move these blocks to the front, preserving their natural order and the coordinate order within each block, followed by the
remaining blocks in their natural order.
This is a permutation of the $t$ individual coordinates, under which $\mathsf P_t$ is unchanged by Lemma~\ref{lem:Ptinv}.

Then identifying $\mathfrak X_J$ with $\mathfrak X^j$ by concatenation, the definition of partial trace and
Theorem~\ref{thm:compatible-projections} together yield
\[
 \Tr_{\overline{J}}(m_t^{-1}\mathsf P_t)
 =\Tr_{t-j}(m_t^{-1}\mathsf P_t)
 =m_j^{-1}\mathsf P_j.
\]
This reduced density matrix has exactly $m_j$ non-zero eigenvalues, all equal to $1/m_j$.
Hence, $S_J=\log m_j$.

Substitute the expressions for $S_J$ into \eqref{eq:strong-subadditivity} and \eqref{eq:weak-monotonicity} to obtain
\begin{align*}
 \log m_{a+b+c}+\log m_b
 &\leqslant\log m_{a+b}+\log m_{b+c},\\
 \log m_b+\log m_c
 &\leqslant\log m_{a+b}+\log m_{a+c}.
\end{align*}
Exponentiating completes the proof.
\end{proof}

\section*{Acknowledgment}

ChatGPT 5.6 led us to the entropy inequalities of Lieb and Ruskai.
AI tools subsequently assisted with proofreading.
All AI-assisted outputs were independently checked, verified, and simplified by the authors, who assume full responsibility for the content of this work.

\end{document}